\documentclass[preprint,11pt,a4paper]{elsarticle}
\usepackage[T1]{fontenc}
\usepackage[utf8]{inputenc}
\usepackage{lmodern}
\usepackage{amsmath,amssymb,amsthm,mathtools}
\usepackage[margin=27mm,headheight=14pt]{geometry}
\usepackage{microtype}
\biboptions{numbers,sort&compress}
\usepackage{hyperref}
\hypersetup{hidelinks,pdftitle={Critical one-component regularity criteria for the three-dimensional Navier--Stokes equations},pdfauthor={},pdfsubject={One-component regularity; Navier--Stokes equations; critical Sobolev and Besov spaces}}
\numberwithin{equation}{section}
\newtheorem{theorem}{Theorem}[section]
\newtheorem{proposition}[theorem]{Proposition}
\newtheorem{lemma}[theorem]{Lemma}
\theoremstyle{remark}
\newtheorem{remark}[theorem]{Remark}
\newcommand{\R}{\mathbb R}
\newcommand{\Z}{\mathbb Z}
\newcommand{\dd}{\,\mathrm d}
\newcommand{\eps}{\varepsilon}
\newcommand{\diver}{\operatorname{div}}
\newcommand{\curl}{\operatorname{curl}}
\newcommand{\sgn}{\operatorname{sgn}}

\newcommand{\Ltwo}{L^2(\R^3)}
\newcommand{\Lq}{L^q(\R^3)}
\newcommand{\Lthree}{L^3(\R^3)}
\newcommand{\Hs}[1]{\dot H^{#1}(\R^3)}
\newcommand{\Hh}{\dot H^{-\delta,0}(\R^3)}
\newcommand{\Bs}{\dot B^{3/2}_{2,r}(\R^3)}
\newcommand{\mixed}[2]{L^{#1}(\R_{x_3};L^{#2}(\R^2_{x_h}))}
\newcommand{\pair}[2]{\left\langle #1,#2\right\rangle_{L^2(\R^3)}}
\newcommand{\pairh}[2]{\left\langle #1,#2\right\rangle_{\Hh}}
\newcommand{\norm}[2]{\left\lVert #1\right\rVert_{#2}}

\newcommand{\doi}[1]{\href{https://doi.org/#1}{\nolinkurl{doi:#1}}}
\newcommand{\arxiv}[1]{\href{https://arxiv.org/abs/#1}{\nolinkurl{arXiv:#1}}}
\allowdisplaybreaks[2]
\journal{Journal of Differential Equations}

\begin{document}
\begin{frontmatter}
\title{A Critical one-component regularity criteria for the three-dimensional Navier--Stokes equations}

\author{
Maotuo Guo\footnote{School of Science, Harbin University of Science and Technology},
~~
Shiyang Xiong\footnote{School of Mathematical Sciences, Shanghai Jiao Tong University, Shanghai 200240,
China,P. R. China. 
\textbf{Corresponding author}. Email: xiongshiyang@amss.ac.cn},
~~
Chuankai Zhao\footnote{School of Mathematics, Statistics and Mechanics, Beijing University of Technology, Beijing 100124,
China.  Email: chuankaizhao@amss.ac.cn}
}

\begin{abstract}
We establish two critical one-component regularity criteria for the three-dimensional incompressible Navier--Stokes equations. First, for divergence-free initial data $u_0\in H^1(\mathbb R^3)$, we prove that a strong solution can be continued beyond a finite time $T$ provided that
\[
 u^3\in L^2\bigl(0,T;\dot B^{3/2}_{2,r}(\mathbb R^3)\bigr)
 \qquad\text{for some }2<r\leq\infty.
\]
Second, for every $2<p<\infty$, if $u_0\in\dot H^{1/2}(\mathbb R^3)$ and, for some fixed unit vector $\boldsymbol e\in\mathbb S^2$,
\[
 u\cdot\boldsymbol e
 \in
 L^p\bigl(0,T;\dot H^{1/2+2/p}(\mathbb R^3)\bigr),
\]
then the corresponding Fujita--Kato solution also extends beyond $T$, without any additional low-integrability assumption on the initial vorticity.

These results extend the critical one-component regularity criteria of Han, Lei, Li, and Zhao in two different directions. At the time endpoint $p=2$, we combine a new vertical-vorticity estimate with coupled anisotropic energy estimates and an energy-dependent frequency decomposition to relax the dyadic square-summability condition
\[
 \dot H^{3/2}=\dot B^{3/2}_{2,2}
\]
to
\[
 \dot B^{3/2}_{2,r},
 \qquad 2<r\leq\infty,
\]
with the endpoint case $r=\infty$ closed by an Osgood-type argument. In the nonendpoint range $2<p<\infty$, we use a positive-time heat-flow decomposition to separate the solution into a smooth linear part and a nonlinear remainder. The remainder has finite energy and acquires the required $L^q$-integrability of the vorticity at positive times, which allows the anisotropic energy estimates to be applied while the additional transport, stretching, and pressure terms generated by the heat flow remain integrable. This removes the extra $L^{q_0}$-integrability assumption on the initial vorticity while retaining the natural critical initial-data space $\dot H^{1/2}(\mathbb R^3)$.
\end{abstract}
\end{frontmatter}

\section{Introduction}
\label{sec:intro}
We consider the Cauchy problem for the three-dimensional incompressible Navier--Stokes equations
\begin{equation}\label{eq:NS}
\begin{cases}
\partial_tu+u\cdot\nabla u-\Delta u+\nabla\pi=0,
    &(t,x)\in(0,T)\times\R^3,\\
\diver u=0, &(t,x)\in(0,T)\times\R^3,\\
u|_{t=0}=u_0, &x\in\R^3.
\end{cases}
\end{equation}
Here $u=(u^1,u^2,u^3)$ and $\pi$ denote the velocity and pressure, the viscosity is normalized to one, and the spatial domain is $\R^3$. We study scale-invariant regularity criteria expressed in terms of a single fixed velocity component. The paper has two related objectives: to weaken the spatial frequency summability at the time endpoint $p=2$, and to remove the auxiliary initial-vorticity hypothesis from the nonendpoint critical Sobolev criteria.

For finite-energy initial data, Leray \cite{Leray1934} and Hopf \cite{Hopf1951} constructed weak solutions satisfying
\begin{equation}\label{eq:basic-energy}
\norm{u(t)}{\Ltwo}^{2}
 +2\int_0^t\norm{\nabla u(s)}{\Ltwo}^{2}\dd s
 \leq \norm{u_0}{\Ltwo}^{2}.
\end{equation}
For divergence-free $u_0\in H^1(\R^3)$, the local strong-solution theory gives a unique maximal solution, smooth at positive times, and continuation follows as long as its $H^1$ norm remains bounded; see Fujita and Kato \cite{FK1964}, Kato \cite{Kato1984}, and Giga \cite{Giga1986}. For critical initial data $u_0\in\dot H^{1/2}(\R^3)$, the Fujita--Kato theory gives the corresponding maximal mild solution.

The scaling
\begin{equation}\label{eq:scaling}
u_\lambda(t,x)=\lambda u(\lambda^2t,\lambda x),
\qquad
\pi_\lambda(t,x)=\lambda^2\pi(\lambda^2t,\lambda x)
\end{equation}
preserves \eqref{eq:NS}. The norm of a velocity component in $L^p(0,T;\dot B^s_{m,r})$ has scaling exponent $1+s-3/m-2/p$, and hence the critical relation is
\begin{equation}\label{eq:critical-indices}
s=-1+\frac3m+\frac2p.
\end{equation}
The Besov summation index $r$ does not enter the scaling. Our endpoint result has $p=m=2$, $s=3/2$, and $2<r\leq\infty$. Our nonendpoint result keeps $m=2$ and treats every finite $p>2$ in the critical Sobolev space $\dot H^{1/2+2/p}$. Throughout the paper, \emph{endpoint} refers to the time exponent $p=2$ in this one-component Sobolev family.

\subsection{One-component regularity criteria}
The full-velocity Ladyzhenskaya--Prodi--Serrin condition is
\begin{equation}\label{eq:LPS}
 u\in L^p(0,T;L^m(\R^3)),\qquad
 \frac2p+\frac3m=1,\qquad 3<m\leq\infty;
\end{equation}
see \cite{Ladyzhenskaya1967,Prodi1959,Serrin1962,Giga1986}. The limiting condition $u\in L^\infty(0,T;L^3)$ was established by Escauriaza, Seregin, and \v Sver\'ak \cite{ESS2003}. In a one-component criterion the remaining velocity components are not prescribed and must be recovered through incompressibility and the equations.

Early one-component results include \cite{NP1999,He2002,KZ2006,ZP2010}. Chae and Wolf \cite{CW2021} treated the strictly subcritical condition
\[
 u^3\in L^p(0,T;L^m),\qquad \frac2p+\frac3m<1,\qquad 3<m\leq\infty,
\]
in a finite-energy solution class. Wang, Wu, and Zhang \cite{WWZ2024} obtained the critical time-Lorentz condition
\begin{equation}\label{eq:WWZ}
 u^3\in L^{p,1}(0,T;L^m),\qquad
 \frac2p+\frac3m=1,\qquad 3<m<\infty.
\end{equation}
The time-Lorentz norm in \eqref{eq:WWZ} is different from the ordinary $L^p$ time norm used below.

The anisotropic energy approach uses the scale-invariant quantity
\begin{equation}\label{eq:Ip}
\int_0^T\norm{u(t)\cdot \boldsymbol e}{\Hs{1/2+2/p}}^p\dd t,
\qquad \boldsymbol e\in\mathbb S^2.
\end{equation}
Chemin and Zhang \cite{CZ2016} proved the corresponding criterion for $4<p<6$, assuming initial vorticity in $L^{3/2}$. Chemin, Zhang, and Zhang \cite{CZZ2017} extended the range to $4<p<\infty$ under an $L^{3/2}\cap L^2$ initial-vorticity assumption. Han, Lei, Li, and Zhao \cite{HLLZ2019} treated $2\leq p\leq4$ and also recovered $4<p<\infty$; their theorem assumes $u_0\in\dot H^{1/2}$ together with $\curl u_0\in L^{q_0}$ for some $1<q_0<2$. In particular, their result includes the endpoint condition
\begin{equation}\label{eq:sobolev-endpoint}
u^3\in L^2(0,T;\dot H^{3/2}(\R^3)).
\end{equation}
The low-exponent initial-vorticity assumption is a genuine additional hypothesis and is not supplied by the critical velocity space alone.

The method couples the vertical vorticity $\omega=\partial_1u^2-\partial_2u^1$ with $u^3$. Writing $u^h=(u^1,u^2)$, incompressibility gives
\begin{equation}\label{eq:intro-reconstruction}
u^h=\nabla_h^\perp\Delta_h^{-1}\omega
       -\nabla_h\Delta_h^{-1}\partial_3u^3.
\end{equation}
Thus the horizontal curl and divergence of $u^h$ are determined by $\omega$ and $\partial_3u^3$. The resulting coupled energy estimates balance the horizontal inverse Laplacian by negative horizontal regularity.

Liu and Zhang \cite{LZ2019} obtained criteria involving Lebesgue and anisotropic Besov conditions. Their later strong-solution theorem \cite{LZ2024} permits a piecewise $H^1$ direction $\beta(t)\in\mathbb S^2$ and assumes
\[
 \int_0^T\norm{u(t)\cdot\beta(t)}{\Hs{3/2}}^2\dd t<\infty.
\]
No separate $L^{q_0}$ initial-vorticity condition is imposed in that strong-solution setting. Our endpoint result concerns a fixed direction but weakens the spatial frequency summability from $\ell^2$ to $\ell^r$, including $\ell^\infty$.

Recent preprints consider critical Besov refinements. Luo, Yin, and Zheng \cite{LYZ2026} establish, under an additional initial-vorticity hypothesis, a criterion in
\begin{equation}\label{eq:LYZ}
L^p\bigl(0,T;\dot B^{1/2+2/p}_{2,\infty}(\R^3)\bigr),
\qquad 2<p<\infty.
\end{equation}
Guo, Wang, and Xiong \cite{GWX2026} obtain a criterion for finite-energy suitable weak solutions in
\begin{equation}\label{eq:GWX}
\widetilde L^p\bigl(0,T;\dot B^0_{m,1}(\R^3)\bigr),
\qquad m=\frac{3p}{p-2},\qquad 2<p<\infty.
\end{equation}
The Chemin--Lerner convention in \eqref{eq:GWX} places the time norm inside the dyadic sum. These preprints do not contain the time endpoint $p=2$ considered in our first theorem. Immediately after Theorem~\ref{thm:nonendpoint}, Remark~\ref{rem:besov-extension} records a conditional route by which the heat-flow perturbation might be combined with the fixed-time Besov estimates of \cite{LYZ2026}. The perturbed Besov estimates required to turn that route into a theorem are not proved here.

\subsection{Main results}
Our endpoint condition is
\begin{equation}\label{eq:criterion}
u^3\in L^2\bigl(0,T;\dot B^{3/2}_{2,r}(\R^3)\bigr),
\qquad 2<r\leq\infty.
\end{equation}
For $r>2$, the strict embedding
\[
\dot H^{3/2}(\R^3)=\dot B^{3/2}_{2,2}(\R^3)
 \hookrightarrow\dot B^{3/2}_{2,r}(\R^3)
\]
shows that \eqref{eq:sobolev-endpoint} implies \eqref{eq:criterion}, but not conversely. Thus regularity under the weaker Besov assumption does not follow from the Sobolev criterion alone; the missing square summability is recovered through the coupled dissipation.

\begin{theorem}[Endpoint Besov criterion]\label{thm:main}
Let $u_0\in H^1(\R^3)$ be divergence free, and let $u$ be the maximal strong solution of \eqref{eq:NS} on $[0,T^*)$. If $T^*<\infty$ and, for some $2<r\leq\infty$,
\begin{equation}\label{eq:main-condition}
\int_0^{T^*}\norm{u^3(t)}{\Bs}^{2}\dd t<\infty,
\end{equation}
then $u$ extends as a strong solution beyond $T^*$.
\end{theorem}
In particular, the criterion reaches the weakest spatial summability endpoint $r=\infty$ in the stated family. By rotational invariance, $u^3$ may be replaced by $u\cdot\boldsymbol e$ for any fixed $\boldsymbol e\in\mathbb S^2$. No smallness or low-exponent initial-vorticity assumption is imposed beyond the stated $H^1$ initial data.

\begin{theorem}[Nonendpoint Sobolev criterion without an initial-vorticity hypothesis]
\label{thm:nonendpoint}
Let $u_0\in\dot H^{1/2}(\R^3)$ be divergence free, and let $u$ be its maximal Fujita--Kato mild solution on $[0,T^*)$. Thus
\[
 u\in C([0,T^*);\dot H^{1/2}(\R^3))
       \cap L^2_{\mathrm{loc}}([0,T^*);\dot H^{3/2}(\R^3)).
\]
Let $2<p<\infty$, set $s_p=1/2+2/p$, and fix $\boldsymbol e\in\mathbb S^2$. If $T^*<\infty$ and
\begin{equation}\label{eq:nonendpoint-condition}
 \int_0^{T^*}\norm{u(t)\cdot\boldsymbol e}{\dot H^{s_p}(\R^3)}^p\dd t<\infty,
\end{equation}
then $u$ extends as a Fujita--Kato solution beyond $T^*$. No assumption $\curl u_0\in L^{q_0}(\R^3)$, $1<q_0<2$, and no additional condition $u_0\in L^2(\R^3)$ are imposed.
\end{theorem}

\begin{remark}[A conditional route to the nonendpoint Besov scale]
\label{rem:besov-extension}
The heat-flow decomposition used below is conceptually independent of the particular fixed-time estimate chosen for the principal nonlinear terms. The recent preprint of Luo, Yin, and Zheng \cite{LYZ2026} establishes fixed-time estimates in which the Sobolev factor $\dot H^{s_p}$ is replaced by $\dot B^{s_p}_{2,\infty}$, where $s_p=1/2+2/p$. Formally, their estimates provide a bound of the form
\begin{equation}\label{eq:remark-besov-interface}
 |\mathcal N_\omega|+|\mathcal N_3|
 \leq C\norm{v^3}{\dot B^{s_p}_{2,\infty}}
       \mathcal E_p^{1/p}\mathcal D_p^{1-1/p}.
\end{equation}
If, in addition, the multiplier, transport, and pressure estimates for the smooth heat-flow coefficients are verified in the auxiliary anisotropic spaces used in \cite{LYZ2026}, then the perturbative scheme of Section~\ref{sec:nonendpoint}, together with
\[
 \norm{e^{(t-a)\Delta}u^3(a)}{\dot B^{s_p}_{2,\infty}}
 \leq C(t-a)^{-1/p}\norm{u(a)}{\dot H^{1/2}},\qquad t\geq b>a,
\]
would lead to the analogous continuation criterion under
\[
 u\cdot\boldsymbol e\in
 L^p\bigl(0,T^*;\dot B^{s_p}_{2,\infty}(\R^3)\bigr),
 \qquad 2<p<\infty,
\]
without a separate $L^{q_0}$ assumption on the initial vorticity. The required perturbed Besov energy estimates, particularly those in the $p>4$ anisotropic regime, are not established in the present paper. Accordingly, this paragraph is only a methodological observation and no nonendpoint Besov theorem is asserted here.
\end{remark}

Theorem~\ref{thm:nonendpoint} retains the natural Fujita--Kato initial-velocity class and removes the additional initial-vorticity condition from the critical Sobolev criterion of \cite{HLLZ2019}. The endpoint $p=2$ is treated separately by Theorem~\ref{thm:main}, and no assertion for $p=\infty$ is made.

\paragraph{Contributions and analytic inputs.}
The endpoint theorem replaces the $\ell^2$ summation in $\dot H^{3/2}$ by $\ell^r$, including $\ell^\infty$, while retaining ordinary $L^2$ time integration. Its new ingredients are a logarithm-free vertical-vorticity estimate and an energy-dependent frequency decomposition that uses the coupled dissipation to compensate for the missing dyadic square summability. The nonendpoint theorem uses the fixed-time Sobolev estimates of Han, Lei, Li, and Zhao \cite{HLLZ2019}; the new step is a positive-time heat-flow decomposition that produces the auxiliary $L^q$ vorticity integrability for the nonlinear remainder and controls every additional transport, stretching, and pressure term. The main theorems therefore use fixed-time inputs only from the published work \cite{HLLZ2019}. The recent Besov estimates of \cite{LYZ2026} are mentioned only in the conditional methodological Remark~\ref{rem:besov-extension}; no result in the paper depends on that remark.

\subsection{Outline of the proof}
The endpoint argument couples $v^3$ with the vertical vorticity $\omega=\partial_1v^2-\partial_2v^1$. A horizontal-derivative trace estimate gives
\begin{equation}\label{eq:intro-trace}
\norm{\nabla_hv^3}{\mixed{\infty}{2}}
 \leq C\norm{v^3}{\Hs{3/2}},
\end{equation}
and controls the term containing a horizontal potential of $\partial_3\omega$ without a logarithmic loss. Together with the third-component estimate of \cite{HLLZ2019}, this yields
\[
E'+cD\leq C\norm{v^3}{\Hs{3/2}}E^{1/2}D^{1/2}+gE,
\qquad g\in L^1.
\]
An isotropic frequency decomposition at a level $N\simeq\log E$ absorbs the high frequencies into $D$ before Young's inequality is applied. The middle frequencies contribute $N^{1/2-1/r}$, leading to
\[
E'\lesssim (1+\norm{v^3}{\Bs}^2+g)
 E(\log E)^{1-2/r},
\]
which is closed by Osgood's lemma, including $r=\infty$.

For both the endpoint theorem and the removal of the nonendpoint initial-vorticity condition, we fix $0<a<b<T^*$ and write
\[
 U(t)=e^{(t-a)\Delta}u(a),\qquad v=u-U.
\]
The nonlinear remainder has finite kinetic energy even when the full velocity does not, and a vorticity Duhamel formula gives $\curl v(b)\in L^q$ for some $q<2$. In the nonendpoint Sobolev range, the fixed-time estimates of \cite{HLLZ2019} and the heat-flow error bounds give an ordinary Gronwall inequality. The auxiliary bounds then yield uniform $L^2$ control of the full velocity gradient and $L^2$ control of its dissipation. A terminal-trace lemma then yields a strong $H^1$ limit for the remainder and permits continuation beyond $T^*$.

The paper is organized around the two main proofs. Section~\ref{sec:notation} collects the function-space conventions, the common heat-flow remainder and perturbation estimates, and the full-gradient and terminal-continuation tools. Section~\ref{sec:endpoint} proves the endpoint Besov criterion in Theorem~\ref{thm:main}; Section~\ref{sec:nonendpoint} proves the nonendpoint Sobolev criterion in Theorem~\ref{thm:nonendpoint}. The conditional Besov observation remains separate in Remark~\ref{rem:besov-extension}.

\section{Preliminaries and continuation tools}\label{sec:notation}
The common estimates in this section are independent of the component hypothesis. They will be used with the endpoint Besov condition in Section~\ref{sec:endpoint} and the nonendpoint Sobolev condition in Section~\ref{sec:nonendpoint}.

\subsection{Notation and function spaces}
We write $x=(x_h,x_3)\in\R^2\times\R$, $x_h=(x_1,x_2)$, and
\[
\nabla_h=(\partial_1,\partial_2),\qquad
\nabla_h^\perp=(-\partial_2,\partial_1),\qquad
\Delta_h=\partial_1^2+\partial_2^2.
\]
For a vector field $v$, let $v^h=(v^1,v^2)$ and $\Omega=\curl v$. The scalar $\omega=\Omega^3$ will always denote its vertical vorticity. We use $|\nabla|^s$, $|\nabla_h|^s$, and $|\partial_3|^s$ for the Fourier multipliers with symbols $|\xi|^s$, $|\xi_h|^s$, and $|\xi_3|^s$, respectively. The Fourier transform $\mathcal F$ is unitary on $L^2(\R^d)$.

The homogeneous Sobolev norm is
\[
\norm{f}{\dot H^s(\R^d)}^2
 =\int_{\R^d}|\xi|^{2s}|\widehat f(\xi)|^2\dd\xi.
\]
The inhomogeneous space $H^s(\R^d)$ uses the weight $(1+|\xi|^2)^s$. Following the anisotropic notation in \cite[Definition 2.1]{CZZ2017}, we set
\begin{equation}\label{eq:anisotropic-norm}
\norm{f}{\dot H^{s,s'}(\R^3)}^2
 =\int_{\R^3}|\xi_h|^{2s}|\xi_3|^{2s'}
                  |\widehat f(\xi)|^2\dd\xi.
\end{equation}
All homogeneous spaces are understood with their usual distributional interpretation. In particular,
\[
\norm{f}{\Hh}
 =\norm{|\nabla_h|^{-\delta}f}{\Ltwo},
\qquad
\langle f,g\rangle_{\Hh}
 =\pair{|\nabla_h|^{-\delta}f}{|\nabla_h|^{-\delta}g}.
\]
The exponent $\delta$ is fixed in \eqref{eq:delta}. Thus the two indices in $\dot H^{s,s'}$ always refer to horizontal and vertical regularity, in that order.

For $1\leq a,b<\infty$, the mixed norm is
\begin{equation}\label{eq:mixed-norm}
\norm{f}{\mixed{b}{a}}
 =\left(\int_\R
  \left(\int_{\R^2}|f(x_h,x_3)|^a\dd x_h\right)^{b/a}
 \dd x_3\right)^{1/b},
\end{equation}
with the usual essential-supremum conventions. A horizontal Sobolev norm at fixed $x_3$ is written $\dot H^s(\R^2_{x_h})$. In particular,
\[
\norm{|\nabla_h|^sf}{\Ltwo}
 =\norm{f}{L^2(\R_{x_3};\dot H^s(\R^2_{x_h}))}.
\]
We also write $\norm{f}{a}=\norm{f}{L^a(\R^3)}$ when the spatial domain is unambiguous. Time arguments are suppressed in estimates at a fixed time. Constants $C$ and $c>0$ may change from line to line; their dependence on a fixed exponent is indicated when needed. They do not depend on the frequency cutoff or on time.

\paragraph{Littlewood--Paley decomposition and time norms.}
Choose a smooth radial function $\chi$, equal to one near the origin and supported in a ball, and put $\varphi(\xi)=\chi(\xi/2)-\chi(\xi)$. Then
\[
\sum_{j\in\Z}\varphi(2^{-j}\xi)=1\qquad(\xi\neq0).
\]
The homogeneous dyadic blocks and low-frequency operators on $\R^3$ are
\[
\dot\Delta_jf=\mathcal F^{-1}\bigl(\varphi(2^{-j}\xi)\widehat f(\xi)\bigr),
\qquad
\dot S_jf=\mathcal F^{-1}\bigl(\chi(2^{-j}\xi)\widehat f(\xi)\bigr).
\]
For the Littlewood--Paley and paradifferential calculus, we refer to Bony \cite{Bony1981} and \cite{BCD2011}. The homogeneous Besov norm is
\begin{equation}\label{eq:Besov-definition}
\norm{f}{\dot B^s_{m,r}(\R^3)}
 =\norm{\bigl(2^{js}\norm{\dot\Delta_jf}{L^m(\R^3)}\bigr)_{j\in\Z}}
 {\ell^r(\Z)}.
\end{equation}
When $r=\infty$, the sequence norm is a supremum. For finite-energy fields, the $L^2(\R^3)$ representative fixes the polynomial ambiguity. For the Fujita--Kato solution in Theorem~\ref{thm:nonendpoint}, we instead use its $L^3(\R^3)$ representative, supplied by $\dot H^{1/2}\hookrightarrow L^3$. No $L^2$ representative of the full velocity is assumed.

The time norm in Theorem~\ref{thm:main} is
\begin{equation}\label{eq:Bochner}
\norm{f}{L^2(0,T;\Bs)}
 =\left(\int_0^T
 \norm{\bigl(2^{3j/2}\norm{\dot\Delta_jf(t)}{\Ltwo}\bigr)_{j\in\Z}}
 {\ell^r(\Z)}^2\dd t\right)^{1/2}.
\end{equation}
It differs from the Chemin--Lerner norm \cite{CL1995,BCD2011}
\begin{equation}\label{eq:CL}
\norm{f}{\widetilde L^2(0,T;\Bs)}
 =\norm{\bigl(2^{3j/2}\norm{\dot\Delta_jf}{L^2(0,T;L^2(\R^3))}\bigr)_{j\in\Z}}
 {\ell^r(\Z)}.
\end{equation}
For $r\geq2$, Minkowski's inequality gives
\begin{equation}\label{eq:Bochner-CL-direction}
 \norm{f}{\widetilde L^2(0,T;\dot B^{3/2}_{2,r})}
 \leq \norm{f}{L^2(0,T;\dot B^{3/2}_{2,r})}.
\end{equation}
Thus the distinction made here concerns the order of the time norm and the dyadic summation; the ordinary Bochner condition is not presented as weaker than its Chemin--Lerner analogue. All frequency estimates below are applied at a fixed time before time integration, and the resulting bounds are then integrated in time.

\subsection{Horizontal reconstruction and auxiliary energies}
The letter $r$ is reserved for the Besov summation index. For the common tools, fix $3/2<q<2$ and set
\begin{equation}\label{eq:delta}
\delta=\frac3q-\frac32\in\left(0,\frac12\right).
\end{equation}
The value of $q$ is chosen independently in Sections~\ref{sec:endpoint} and \ref{sec:nonendpoint} to meet the cited spatial estimates, and is then fixed within each proof. It never depends on time or on a frequency cutoff. Unless specified otherwise, the reconstruction below is first applied to a smooth divergence-free field in $H^m(\R^3)$, $m\geq3$. We set
\begin{equation}\label{eq:omega-f}
\omega=\partial_1v^2-\partial_2v^1.
\end{equation}
We use the signed-power notation of \cite[equation (2.9)]{CZZ2017}:
\[
a_\gamma=\sgn(a)|a|^\gamma\qquad(\gamma>0).
\]
The signed and unsigned powers have the same gradient norm. H\"older's inequality gives
\begin{equation}\label{eq:G-W}
\norm{\nabla\omega}{\Lq}
 \leq\frac2q\norm{\omega}{\Lq}^{1-q/2}
              \norm{\nabla\omega_{q/2}}{\Ltwo}.
\end{equation}
Indeed, $|\nabla\omega|=(2/q)|\omega|^{1-q/2}|\nabla\omega_{q/2}|$ almost everywhere. The identity and estimate are justified by regularization at $\omega=0$.

Incompressibility and \eqref{eq:omega-f} give the horizontal Biot--Savart formula
\begin{equation}\label{eq:reconstruction}
v^h=\nabla_h^\perp\Delta_h^{-1}\omega
       -\nabla_h\Delta_h^{-1}\partial_3v^3.
\end{equation}
This is the curl--div decomposition used in \cite{CZ2016,CZZ2017,HLLZ2019}. The inverse horizontal Laplacian and its derivatives are interpreted as homogeneous Fourier multipliers. The identity is first verified for $\xi_h\neq0$. Since $v\in L^2(\R^3)$, its Fourier transform cannot have a nonzero part supported on $\{\xi_h=0\}$, so no additional horizontal harmonic component is present.

We shall also use the following consequence of incompressibility. For a divergence-free $v\in H^1(\R^3)$,
\[
|\xi_3\widehat v^3(\xi)|\leq|\xi_h|\,|\widehat v^h(\xi)|.
\]
Consequently,
\begin{equation}\label{eq:Y-H1}
\begin{split}
\norm{\nabla v^3}{\Hh}^2
 &=\int_{\R^3}|\xi_h|^{-2\delta}|\xi|^2
                         |\widehat v^3(\xi)|^2\dd\xi\\
 &\leq\int_{\R^3}|\xi_h|^{2-2\delta}|\widehat v(\xi)|^2\dd\xi
 \leq C_\delta\norm{v}{H^1(\R^3)}^2.
\end{split}
\end{equation}

Define the auxiliary energy and dissipation by
\begin{equation}\label{eq:ED}
\begin{aligned}
E(t)&=\mathrm e+\norm{\omega(t)}{\Lq}^{2}
          +\norm{\nabla v^3(t)}{\Hh}^{2},\\
D(t)&=\norm{\omega(t)}{\Lq}^{2-q}
          \norm{\nabla\omega_{q/2}(t)}{\Ltwo}^{2}
          +\norm{\nabla^2v^3(t)}{\Hh}^{2}.
\end{aligned}
\end{equation}
Here $\mathrm e=\exp(1)$. These are the auxiliary energy and dissipation for the coupled estimates, not the kinetic energy in \eqref{eq:basic-energy}. Vector and tensor norms are square sums over all components; in particular, $\nabla^2v^3$ includes every ordered pair of derivatives. The first term in $D$ is set to zero when $\norm{\omega}{\Lq}=0$.

For later use, abbreviate
\begin{equation}\label{eq:common-XYZW}
 X=\norm\omega q,\quad
 W=X^{1-q/2}\norm{\nabla\omega_{q/2}}2,\quad
 Y=\norm{\nabla v^3}{\Hh},\quad
 Z=\norm{\nabla^2v^3}{\Hh}.
\end{equation}
Then $E=\mathrm e+X^2+Y^2$, $D=W^2+Z^2$, and $\norm{\nabla\omega}q\leq(2/q)W$. Given a smooth divergence-free field $w$, let $\pi_w$ be its Riesz pressure,
\begin{equation}\label{eq:Pi}
 -\Delta\pi_w=\sum_{i,j=1}^3\partial_iw^j\,\partial_jw^i,
\end{equation}
and define the two principal nonlinear pairings by
\begin{equation}\label{eq:common-pairings}
\begin{aligned}
 \mathcal N_\omega(w)
 &=X^{2-q}\int_{\R^3}((\curl w)\cdot\nabla w^3)\omega_{q-1}\dd x,\\
 \mathcal N_3(w)
 &=\sum_{k=1}^3\big\langle
 \partial_k(w\cdot\nabla w^3+\partial_3\pi_w),
 \partial_kw^3\big\rangle_{\Hh},
\end{aligned}
\end{equation}
where $X,\omega$ in this formula are formed from $w$. The value of $\mathcal N_\omega$ is set to zero when $X=0$.

\subsection{Multipliers and the common heat-flow remainder}\label{subsec:common-heat}
We first record the multiplier bounds used to control the smooth heat-flow coefficients. 
\begin{lemma}\label{lem:products}
Let $0<\delta<1$. For smooth $f,g$ with finite right-hand sides,
\begin{equation}\label{eq:product1}
\norm{fg}{\Hh}
 \leq C_\delta\norm{f}{\mixed{\infty}{2/\delta}}
                 \norm{g}{\Ltwo}.
\end{equation}
On each horizontal plane,
\begin{equation}\label{eq:product2}
\norm{fg}{\dot H^{-\delta}(\R^2)}
 \leq C_\delta\bigl(\norm{f}{L^\infty(\R^2)}
                   +\norm{\nabla_h f}{L^2(\R^2)}\bigr)
                   \norm{g}{\dot H^{-\delta}(\R^2)}.
\end{equation}
Estimate \eqref{eq:product1} extends to $f\in\mixed{\infty}{2/\delta}$ and $g\in L^2(\R^3)$. In \eqref{eq:product2}, multiplication extends by duality to $f\in L^\infty(\R^2)\cap\dot H^1(\R^2)$ and $g\in\dot H^{-\delta}(\R^2)$. The corresponding multiplication operator on $\dot H^{-\delta,0}(\R^3)$ is bounded by the essential supremum in $x_3$ of the coefficient in \eqref{eq:product2}.
\end{lemma}
\begin{proof}
The dual horizontal Sobolev inequality \cite{Stein1970} and H\"older's inequality give, at fixed $x_3$,
\[
\norm{fg}{\dot H^{-\delta}(\R^2)}
 \leq C_\delta\norm{fg}{L^{2/(1+\delta)}(\R^2)}
 \leq C_\delta\norm{f}{L^{2/\delta}(\R^2)}\norm{g}{L^2(\R^2)}.
\]
Squaring and integrating in $x_3$ proves \eqref{eq:product1}.

For \eqref{eq:product2}, use duality with $\dot H^\delta(\R^2)$ and the fractional product inequality; see \cite{GO2014,BCD2011}. For a test function $\phi$,
\begin{align*}
\norm{f\phi}{\dot H^\delta(\R^2)}
&\leq C_\delta\Bigl(
 \norm{f}{L^\infty(\R^2)}\norm{\phi}{\dot H^\delta(\R^2)}
 +\norm{|\nabla_h|^\delta f}{L^{2/\delta}(\R^2)}
  \norm{\phi}{L^{2/(1-\delta)}(\R^2)}\Bigr)\\
&\leq C_\delta\bigl(\norm{f}{L^\infty(\R^2)}
                +\norm{\nabla_h f}{L^2(\R^2)}\bigr)
                \norm{\phi}{\dot H^\delta(\R^2)}.
\end{align*}
We used
\[
\dot H^\delta(\R^2)\hookrightarrow L^{2/(1-\delta)}(\R^2),
\qquad
\norm{|\nabla_h|^\delta f}{L^{2/\delta}(\R^2)}
 \leq C_\delta\norm{\nabla_h f}{L^2(\R^2)}.
\]
The bounded map $\phi\mapsto f\phi$ on $\dot H^\delta(\R^2)$ defines multiplication on its dual by $\langle fg,\phi\rangle=\langle g,f\phi\rangle$. This proves the plane-wise bound and its stated extension. Its three-dimensional consequence follows by integration in $x_3$ with a uniform plane-wise multiplier bound.
\end{proof}

\begin{lemma}[Finite-energy heat-flow remainder]\label{lem:common-remainder}
Let $u$ be a Fujita--Kato solution on $[0,T^*)$, with $T^*<\infty$, and fix $0<a<b<T^*$. Set
\begin{equation}\label{eq:heat-decomposition}
 U(t)=e^{(t-a)\Delta}u(a),\qquad v(t)=u(t)-U(t),\qquad
 \Omega_U=\curl U,\quad \omega_U=\Omega_U^3.
\end{equation}
For every $b<\tau<T^*$ and integer $m\geq0$,
\begin{equation}\label{eq:np-local-remainder}
 \begin{gathered}
 v\in C([b,\tau];H^m),\qquad
 \curl v\in C([b,\tau];W^{m,q})\cap C^1([b,\tau];L^q).
 \end{gathered}
\end{equation}
Moreover,
\begin{equation}\label{eq:np-basic-remainder}
 \sup_{b\leq t<T^*}\norm{v(t)}2^2+
 \int_b^{T^*}\norm{\nabla v(t)}2^2\dd t<\infty.
\end{equation}
The quantities $E,D$ formed from $v$ satisfy $E(b)<\infty$, $E\in C^1([b,\tau])$, and $D\in L^1(b,\tau)$. No $L^2$ hypothesis on the full initial velocity is required.
\end{lemma}
\begin{proof}
On compact positive-time intervals, Fujita--Kato smoothing gives $u\in C(W^{m,3})$ and $\curl u\in C(H^m)$ for every $m\geq0$; see \cite{FK1964,Kato1984,BCD2011}. The mild equation starting at $a$ gives
\begin{equation}\label{eq:np-velocity-duhamel}
 v(t)=-\int_a^t e^{(t-s)\Delta}\mathbb P\diver(u\otimes u)(s)\dd s,
\end{equation}
where $\mathbb P$ is the Leray projector. The heat estimate from $L^{3/2}$ to $L^2$, with one derivative, yields
\begin{equation}\label{eq:np-L2-gain}
 \norm{v(t)}2\leq C\int_a^t(t-s)^{-3/4}\norm{u(s)}3^2\dd s<\infty.
\end{equation}
Placing spatial derivatives on $u\otimes u$ gives the asserted local $H^m$ regularity and continuity. Only compact-time bounds are used at this stage.

Put $\Omega_u=\curl u$ and $\ell=2q/(2-q)>3$, so $1/q=1/\ell+1/2$. Since $v(a)=0$, the vorticity equation gives
\begin{equation}\label{eq:Duhamel}
 \curl v(t)=\int_a^t e^{(t-s)\Delta}\diver
 (u\otimes\Omega_u-\Omega_u\otimes u)(s)\dd s.
\end{equation}
We use $(A\otimes B)_{ij}=A^iB^j$ and $(\diver M)^i=\sum_j\partial_jM_{ij}$; thus the divergence equals $\Omega_u\cdot\nabla u-u\cdot\nabla\Omega_u$. Consequently,
\begin{equation}\label{eq:initial-X}
 \norm{\curl v(b)}q\leq C\int_a^b(b-s)^{-1/2}
 \norm{u(s)}\ell\norm{\Omega_u(s)}2\dd s<\infty.
\end{equation}
For higher derivatives, Leibniz' rule gives, on $[a,\tau]$,
\[
 \norm{u\otimes\Omega_u-\Omega_u\otimes u}{W^{m,q}}
 \leq C_m\sum_{|\alpha|+|\gamma|\leq m}
 \norm{\partial^\alpha u}\ell\norm{\partial^\gamma\Omega_u}2.
\]
Passing derivatives in \eqref{eq:Duhamel} to this product leaves the integrable kernel $(t-s)^{-1/2}$ and proves continuity in $W^{m,q}$. Its equation then gives $\partial_t\curl v\in C([b,\tau];L^q)$.

The fixed positive gap $b-a$ supplies all background bounds without assuming $u(a)\in L^2$. Indeed,
\begin{align}
 \norm{\nabla^mU(t)}3
 &\leq C_m(t-a)^{-m/2}\norm{u(a)}3,
 &&0\leq m\leq8,\label{eq:np-heat-L3}\\
 \norm{\nabla^kU(t)}2
 &\leq C_k(t-a)^{-(k-1/2)/2}\norm{u(a)}{\dot H^{1/2}},
 &&1\leq k\leq7.\label{eq:np-heat-L2-derivatives}
\end{align}
The first bound uses $\dot H^{1/2}\hookrightarrow L^3$; the second follows by factoring $|\xi|^{1/2}$ from the Fourier multiplier. Hence
\begin{equation}\label{eq:np-heat-coefficients}
 M:=1+\sup_{b\leq t\leq T^*+1}
 \left(\norm{U(t)}{W^{8,3}}+\norm{\nabla U(t)}{H^6}\right)
 \leq C_{a,b,T^*}\bigl(1+\norm{u(a)}{\dot H^{1/2}}\bigr)<\infty.
\end{equation}
For $|\alpha|\leq2$, these bounds imply
\begin{equation}\label{eq:np-heat-mixed-coefficients}
\begin{split}
 &\norm{\partial^\alpha U}{L^\ell}
 +\norm{\partial^\alpha U}{\mixed{\infty}{2/\delta}}
 +\norm{\partial^\alpha U}{L^\infty}
 +\norm{\nabla_h\partial^\alpha U}{\mixed{\infty}{2}}
 \leq C_qM.
\end{split}
\end{equation}
For the mixed $L^{2/\delta}_{x_h}$ norm, use the horizontal $W^{2,3}$ embedding and the vertical $W^{1,3}$ embedding. For the last norm, use the vertical $H^1$ embedding with values in $L^2_{x_h}$ and \eqref{eq:np-heat-L2-derivatives}. These bounds hold uniformly on $[b,T^*+1]$.

The remainder satisfies
\begin{equation}\label{eq:np-remainder-equation}
 \partial_tv-\Delta v+v\cdot\nabla v+U\cdot\nabla v
 +v\cdot\nabla U+U\cdot\nabla U+\nabla\pi=0.
\end{equation}
Testing by $v$ on $[b,\tau]$ and using $\diver v=\diver U=0$ yields
\begin{align}
 \frac12\frac{\mathrm d}{\mathrm dt}\norm v2^2+\norm{\nabla v}2^2
 &\leq\norm{\nabla U}\infty\norm v2^2+
orm U4^2\norm{\nabla v}2\notag\\
 &\leq\frac12\norm{\nabla v}2^2+C_M(1+\norm v2^2).
 \label{eq:np-basic-differential}
\end{align}
The $L^4$ bound for $U$ follows from its $W^{8,3}$ bound. Gronwall's inequality has constants independent of $\tau<T^*$ and proves \eqref{eq:np-basic-remainder}.

Finally, \eqref{eq:initial-X} and \eqref{eq:Y-H1} give $E(b)<\infty$. On compact intervals, the regularized $L^q$ diffusion identity \cite[Lemma 2.1]{HLLZ2019} gives
\[
 X^{2-q}\norm{\nabla\omega_{q/2}}2^2
 \leq C_q X\norm{\Delta\omega}q.
\]
Thus the first part of $D$ is locally integrable. Applying \eqref{eq:Y-H1} to the divergence-free fields $\partial_jv$ and $\partial_tv$ controls the remaining part of $D$ and the derivative of $Y^2$. Here $\partial_tv\in C([b,\tau];H^1)$ follows from \eqref{eq:np-remainder-equation}, the local smoothing, and the $L^2$ boundedness of $\mathbb P$. The $L^q$ norm-square chain rule then gives $E\in C^1([b,\tau])$.
\end{proof}
The lemma applies in particular to the $H^1$ strong solution in Theorem~\ref{thm:main}. Local smoothing constants may depend on $\tau$, whereas $M$ and the bound in \eqref{eq:np-basic-remainder} do not.

\begin{lemma}[Common perturbation estimate]\label{lem:common-perturbation}
With the notation of Lemma~\ref{lem:common-remainder}, let $E,D$ be formed from $v$. There is a nonnegative $g\in L^1(b,T^*)$ such that
\begin{equation}\label{eq:common-perturbed-energy}
 E'+c_qD\leq
 2\bigl(|\mathcal N_\omega(v)|+|\mathcal N_3(v)|\bigr)+gE.
\end{equation}
One may take
\begin{equation}\label{eq:common-g}
 g(t)=C_{q,M}\bigl(1+\norm{v(t)}2^2+\norm{\nabla v(t)}2^2\bigr).
\end{equation}
\end{lemma}
\begin{proof}
Subtracting the heat equation from the vorticity equation gives
\begin{equation}\label{eq:np-perturbed-scalar}
 (\partial_t+u\cdot\nabla-\Delta)\omega
   =(\curl v)\cdot\nabla v^3+f_\omega,
\end{equation}
where
\begin{equation}\label{eq:np-scalar-errors}
 f_\omega=(\curl v)\cdot\nabla U^3
 +\Omega_U\cdot\nabla(v^3+U^3)
 -(v+U)\cdot\nabla\omega_U.
\end{equation}
The coefficient bounds \eqref{eq:np-heat-mixed-coefficients} and $1/q=1/\ell+1/2$ imply
\begin{equation}\label{eq:np-scalar-error-bound}
 \norm{f_\omega}q\leq C_{q,M}(1+\norm v2+\norm{\nabla v}2).
\end{equation}
Each mixed product pairs an $L^2$ factor from $v$ or $\nabla v$ with an $L^\ell$ heat coefficient. For pure heat products, one differentiated factor is in $L^2$ by \eqref{eq:np-heat-L2-derivatives}. Testing \eqref{eq:np-perturbed-scalar} by $\omega_{q-1}$ and multiplying by $X^{2-q}$ gives the principal contribution $\mathcal N_\omega(v)$ and the error
\begin{equation}\label{eq:np-scalar-tested}
 X\norm{f_\omega}q\leq C X^2+
 C_{q,M}(1+\norm v2^2+\norm{\nabla v}2^2).
\end{equation}
The transport cancels because $u$ is divergence free. The test is justified by regularization and Lemma~\ref{lem:common-remainder}.

For the third component, split $\pi=\pi_v+\pi_1$. Expanding the pressure source gives
\begin{equation}\label{eq:np-pressure-error}
 -\Delta\pi_1=2\sum_{i,j}\partial_i v^j\partial_jU^i
                 +\sum_{i,j}\partial_iU^j\partial_jU^i,
\end{equation}
and the remainder equation is
\begin{equation}\label{eq:np-third-error-equation}
 \partial_tv^3-\Delta v^3+v\cdot\nabla v^3+\partial_3\pi_v
 =-U\cdot\nabla v^3-(v+U)\cdot\nabla U^3-\partial_3\pi_1.
\end{equation}
For $k=1,2,3$, Lemma~\ref{lem:products} gives
\begin{align}
 \norm{\partial_k(U\cdot\nabla v^3)}{\Hh}
 &\leq C_{q,M}(Z+Y),\label{eq:np-U-transport}\\
 \norm{\partial_k((v+U)\cdot\nabla U^3)}{\Hh}
 &\leq C_{q,M}(1+\norm v2+\norm{\nabla v}2),\label{eq:np-U-force}\\
 \norm{\partial_k\partial_3\pi_1}{\Hh}
 &\leq C_{q,M}(1+\norm{\nabla v}2).\label{eq:np-U-pressure}
\end{align}
Indeed, $U$ and $\partial_kU$ are multipliers on $\dot H^{-\delta,0}$ by \eqref{eq:product2}; this proves \eqref{eq:np-U-transport}. After differentiating $(v+U)\cdot\nabla U^3$, estimate \eqref{eq:product1} applies to the factors containing $v$ or $\nabla v$, and differentiated heat factors control the remaining terms. For the pressure, the bounded multiplier $\partial_k\partial_3(-\Delta)^{-1}$ commutes with $|\nabla_h|^{-\delta}$; apply \eqref{eq:product1} to each product in \eqref{eq:np-pressure-error}.

Pairing the differentiated equation with $\partial_kv^3$ in $\dot H^{-\delta,0}$ and summing gives the principal term $-\mathcal N_3(v)$. The additional pairings satisfy
\begin{equation}\label{eq:np-errors-paired}
 |\text{additional pairings}|
 \leq\eps Z^2+C_{q,M,\eps}
       (Y^2+1+\norm v2^2+\norm{\nabla v}2^2).
\end{equation}
Adding the two energy identities, choosing $\eps$ sufficiently small, and using $E\geq\mathrm e$ proves \eqref{eq:common-perturbed-energy}. The coefficient in \eqref{eq:common-g} is integrable by \eqref{eq:np-basic-remainder}.
\end{proof}

\subsection{Full-gradient estimates and terminal continuation}\label{subsec:continuation-tools}
\begin{lemma}[Differentiated energy identity in the homogeneous class]
\label{lem:homogeneous-gradient-identity}
Let $J\Subset(0,T^*)$ and let $u$ be a divergence-free solution of \eqref{eq:NS}, with its Riesz pressure, smooth on $J$, such that
\[
 u\in C(J;L^3(\R^3)),\qquad
 \nabla u\in C(J;L^2(\R^3)),\qquad
 \Delta u\in L^2(J;L^2(\R^3)).
\]
Then the identity
\begin{equation}\label{eq:full-identity}
\frac12\frac{\mathrm d}{\mathrm dt}\norm{\nabla u}{\Ltwo}^2
 +\norm{\Delta u}{\Ltwo}^2
 =-\sum_{i,j,k=1}^3\int_{\R^3}
      \partial_ku^i\,\partial_i u^j\,\partial_ku^j\dd x.
\end{equation}
holds in the time-integrated sense on $J$, and hence almost everywhere in time. No $L^2$ bound for $u$ itself is required.
\end{lemma}
\begin{proof}
Let $\chi_R(x)=\chi(x/R)$, where $\chi\in C_c^\infty(\R^3)$ equals one on the unit ball and is supported in the ball of radius two. Differentiate the equation by $\partial_k$, multiply the $j$th component by $\chi_R^2\partial_ku^j$, and sum over $j,k$. The differentiated transport equation is
\[
 \partial_t\partial_ku^j+u^i\partial_i\partial_ku^j
 +\partial_ku^i\partial_i u^j-\Delta\partial_ku^j
 +\partial_k\partial_j\pi=0.
\]
All terms are integrable on compact time intervals. Indeed,
\[
 \norm{\nabla u}{3}^2
 \leq C\norm{\nabla u}{2}\norm{\Delta u}{2},
 \qquad
 \norm u6\leq C\norm{\nabla u}2,
\]
and, more explicitly,
\[
 \int_J\norm{\nabla u}3^3\dd t
 \leq C|J|^{1/4}\norm{\nabla u}{L^\infty(J;L^2)}^{3/2}
                  \norm{\Delta u}{L^2(J;L^2)}^{3/2}<\infty.
\]
The remaining transport and diffusion products are likewise integrable by H\"older's inequality. Moreover,
\[
 \nabla\pi=\nabla(-\Delta)^{-1}\diver(u\cdot\nabla u),
\]
so Calder\'on--Zygmund estimates, $u\in L^6$, and $\nabla u\in L^2$ give
$\nabla\pi\in L^{3/2}$ locally in time, while $\nabla u\in L^3$.

The transport term equals a boundary integral supported in
$\{R<|x|<2R\}$:
\[
 \frac12\int u\cdot\nabla(\chi_R^2)|\nabla u|^2\dd x,
\]
and tends to zero as $R\to\infty$. The boundary terms generated by the diffusion and pressure are bounded, respectively, by tails of
$|\nabla u||\nabla^2u|$ and $|\nabla\pi||\nabla u|$, multiplied by $C/R$, and therefore also vanish. Passing to the limit gives
\[
 \frac12\frac{\mathrm d}{\mathrm dt}\norm{\nabla u}2^2
 +\norm{\Delta u}2^2
 =-\sum_{i,j,k}\int
   \partial_ku^i\,\partial_i u^j\,\partial_ku^j\dd x
\]
in the time-integrated sense. The same cutoff argument justifies the subsequent integrations by parts in the cubic term.
\end{proof}

\begin{lemma}[Full-gradient estimate]\label{lem:full-gradient}
Under the hypotheses of Lemma~\ref{lem:homogeneous-gradient-identity}, put $\omega_u=\partial_1u^2-\partial_2u^1$. If $\omega_u$ and $\nabla u^3$ belong to $L^2(J;L^3)$, then
\begin{equation}\label{eq:full-energy}
 \frac{\mathrm d}{\mathrm dt}\norm{\nabla u}2^2+\norm{\Delta u}2^2
 \leq C\bigl(\norm{\omega_u}3^2+\norm{\nabla u^3}3^2\bigr)
                 \norm{\nabla u}2^2.
\end{equation}
In particular, if these coefficient norms are integrable up to a finite terminal time, Gronwall's inequality bounds the full gradient and its dissipation up to that time.
\end{lemma}
\begin{proof}
Every cubic monomial on the right contains a factor from either $\nabla_hu^h$ or $\nabla u^3$. Indeed, if $i,j\in\{1,2\}$, then $\partial_i u^j$ is horizontal. If $j=3$, a differentiated $u^3$ appears explicitly. If $i=3$ and $j\in\{1,2\}$, then $\partial_ku^3$ is explicit. Thus
\begin{equation}\label{eq:cubic-pointwise}
\left|\sum_{i,j,k=1}^3
 \partial_ku^i\,\partial_i u^j\,\partial_ku^j\right|
 \leq C\bigl(|\nabla_hu^h|+|\nabla u^3|\bigr)|\nabla u|^2.
\end{equation}
The reconstruction \eqref{eq:intro-reconstruction} and horizontal Riesz-transform boundedness \cite{Stein1970}, applied plane by plane, imply
\begin{equation}\label{eq:horizontal-L3}
\norm{\nabla_hu^h}{\Lthree}
 \leq C\bigl(\norm{\omega_u}{\Lthree}
                  +\norm{\partial_3u^3}{\Lthree}\bigr).
\end{equation}
The differentiated reconstruction is a horizontal order-zero Fourier multiplier. To justify it in the homogeneous setting, first note that $\nabla u\in L^2$ and $u\in L^3$ rule out a distribution supported on $\{\xi_h=0\}$ in the differentiated field. Thus the identity holds for the $L^2$ derivatives. Whenever $\omega_u,\partial_3u^3\in L^3$, horizontal Riesz-transform boundedness and Fubini's theorem give \eqref{eq:horizontal-L3}. No estimate of the undifferentiated horizontal potential in $L^3$ is needed.
Combining \eqref{eq:cubic-pointwise}--\eqref{eq:horizontal-L3} with \eqref{eq:full-identity}, using
$\norm{\nabla u}{L^6(\R^3)}\leq C\norm{\Delta u}{\Ltwo}$,
and then applying Young's inequality proves \eqref{eq:full-energy}.
\end{proof}

\begin{lemma}[Strong terminal trace and continuation]\label{lem:terminal-trace}
Let $a<b<T<\infty$, let $h\in\dot H^{1/2}(\R^3)$ be divergence free, and set $U(t)=e^{(t-a)\Delta}h$ for $t>a$. Suppose that a divergence-free field $v$, smooth on compact subintervals of $[b,T)$, solves
\begin{equation}\label{eq:trace-remainder}
 \partial_tv-\Delta v+\mathbb P\bigl((v+U)\cdot\nabla(v+U)\bigr)=0
\end{equation}
and satisfies
\begin{equation}\label{eq:trace-hypotheses}
 v\in L^\infty(b,T;H^1)\cap L^2(b,T;H^2),\qquad v(b)\in H^1.
\end{equation}
Then $v$ has a strong limit in $H^1$ as $t\uparrow T$ and extends as a strong solution of \eqref{eq:trace-remainder} beyond $T$. Consequently, $u=U+v$ extends in the Fujita--Kato class. If $h\in L^2$ as well, the extension is also a finite-energy $H^1$ strong solution.
\end{lemma}
\begin{proof}
The positive gap $b-a$ and the heat estimates imply, on $[b,T+1]$,
\[
 \norm{U}{\infty}+\norm{\nabla U}{\infty}
 +\norm{\nabla U}{2}+\norm{U\cdot\nabla U}{2}\leq C_{h,a,b,T}.
\]
For the nonlinear term, Sobolev embedding gives
\[
 \norm{v\cdot\nabla v}{2}
 \leq\norm{v}{6}\norm{\nabla v}{3}
 \leq C\norm{v}{H^1}\norm{v}{H^2}.
\]
The mixed terms are bounded by $C_{h,a,b,T}\norm{v}{H^1}$. Since $\mathbb P$ is bounded on $L^2$, \eqref{eq:trace-remainder} and \eqref{eq:trace-hypotheses} imply
\begin{equation}\label{eq:time-derivative-L2}
 \partial_tv\in L^2(b,T;L^2).
\end{equation}
Here the finite-energy object is $v$, not necessarily $u$.

We give the terminal trace argument explicitly. Let $S_N$ be a smooth inhomogeneous Fourier cutoff to $|\xi|\lesssim2^N$, and put $w_N=(I-S_N)v$. Since the cutoff commutes with $\partial_t$, one has, in the sense of distributions in time,
\[
 \frac{\mathrm d}{\mathrm dt}\norm{w_N}{H^1}^2
 =2\big\langle\partial_tw_N,(1-\Delta)w_N\big\rangle_{L^2}.
\]
Integration and Cauchy--Schwarz therefore give
\begin{equation}\label{eq:trace-tail}
 \begin{split}
 \sup_{b\leq t<T}\norm{w_N(t)}{H^1}^2
 &\leq\norm{w_N(b)}{H^1}^2\\
 &\quad+2\norm{(I-S_N)\partial_tv}{L^2(b,T;L^2)}
             \norm{(I-S_N)v}{L^2(b,T;H^2)}.
 \end{split}
\end{equation}
The three quantities on the right tend to zero as $N\to\infty$: the first by $v(b)\in H^1$, and the other two by strong convergence of the Fourier cutoffs in the indicated Bochner spaces. For each fixed $N$, the finite-frequency estimate
$\norm{S_Nf}{H^1}\leq C_N\norm f2$ and \eqref{eq:time-derivative-L2} imply
$S_Nv\in H^1(b,T;H^1)$, hence $S_Nv$ extends continuously to $[b,T]$ in $H^1$. The uniform high-frequency estimate \eqref{eq:trace-tail} shows that these continuous low-frequency parts form a Cauchy sequence in $C([b,T];H^1)$. Thus $v$ has a strong $H^1$ limit at $T$.

Starting with $v(T)$, the local $H^1$ construction applies to \eqref{eq:trace-remainder} with its prescribed smooth coefficients. The usual Galerkin construction or mild iteration is unchanged by these lower-order terms. In particular, its a priori estimate is
\[
 \frac{\mathrm d}{\mathrm dt}\bigl(1+\norm{v}{H^1}^2\bigr)
 +c\norm{\nabla v}{H^1}^2
 \leq C_{h,a,b,T}\bigl(1+\norm{v}{H^1}^2\bigr)^3.
\]
The self-interaction is bounded by
$\norm{v}{6}\norm{\nabla v}{3}\norm{\Delta v}{2}$; all terms involving $U$ use the heat bounds above and their differentiated versions. The corresponding difference estimate gives strong uniqueness. This gives a local continuation, as in the standard $H^1$ theory \cite{FK1964,Giga1986,BCD2011}. Since $U$ is a heat flow from $\dot H^{1/2}$ and $H^1\hookrightarrow\dot H^{1/2}$, the sum lies in the Fujita--Kato class and agrees with the original solution. When $h\in L^2$, $U$ also stays in $H^1$ for $t\geq b$.
\end{proof}

\section{The endpoint Besov criterion}\label{sec:endpoint}
This section proves Theorem~\ref{thm:main}. Choose $q<2$ sufficiently close to $2$ for Lemma~\ref{lem:HLLZ} and keep $q,\delta$ fixed throughout. We first establish the spatial estimates, then use the common perturbation lemma and the endpoint frequency decomposition.

\subsection{Coupled vorticity and third-component estimates}\label{subsec:endpoint-coupled}
\paragraph{Horizontal trace and potential estimates.}
\begin{lemma}\label{lem:trace}
For every $f\in\dot H^{3/2}(\R^3)$,
\begin{equation}\label{eq:trace}
\norm{\nabla_hf}{\mixed{\infty}{2}}
 \leq C\norm{f}{\Hs{3/2}}.
\end{equation}
\end{lemma}
\begin{proof}
We first take $f$ to be a Schwartz function. Fix $x_3\in\R$. By horizontal Plancherel and weighted Cauchy--Schwarz in $\xi_3$,
\begin{align*}
\norm{\nabla_hf(\cdot,x_3)}{L^2(\R^2)}^2
&\leq C\int_{\R^2}|\xi_h|^2
       \left(\int_\R|\widehat f(\xi_h,\xi_3)|\dd\xi_3\right)^2\dd\xi_h\\
&\leq C\int_{\R^2}
 \left[\int_\R(|\xi_h|^2+\xi_3^2)^{3/2}
            |\widehat f(\xi_h,\xi_3)|^2\dd\xi_3\right]\\
&\hspace{4.5em}\times
 \left[\int_\R\frac{|\xi_h|^2}{(|\xi_h|^2+\xi_3^2)^{3/2}}\dd\xi_3\right]\dd\xi_h.
\end{align*}
For $\xi_h\neq0$, a change of variable gives
\[
\int_\R\frac{|\xi_h|^2}{(|\xi_h|^2+\xi_3^2)^{3/2}}\dd\xi_3
 =\int_\R\frac{\dd s}{(1+s^2)^{3/2}}=2.
\]
The bound is independent of $x_3$, proving \eqref{eq:trace}. The general case follows by density in the homogeneous Sobolev space, applied to $\nabla_h f$.
\end{proof}
\begin{remark}\label{rem:mixed}
The left-hand side of \eqref{eq:trace} is
\[
\operatorname*{ess\,sup}_{x_3\in\R}
\left(\int_{\R^2}|\nabla_hf(x_h,x_3)|^2\dd x_h\right)^{1/2}.
\]
It is not $\norm{\nabla_h f}{L^2(\R^2_{x_h};L^\infty(\R_{x_3}))}$. No endpoint embedding of $\dot H^{1/2}(\R)$ into $L^\infty(\R)$ is used. The horizontal derivative in \eqref{eq:trace} is what makes the Fourier integral finite.
\end{remark}

\begin{lemma}\label{lem:difficult}
For $i,j\in\{1,2\}$ and smooth $v^3,\omega$ with finite right-hand side,
\begin{equation}\label{eq:difficult}
\begin{split}
\left|\int_{\R^3}\partial_i v^3\,
 \partial_j\Delta_h^{-1}\partial_3\omega\,
 \omega|\omega|^{q-2}\dd x\right|
 \leq C_q\norm{v^3}{\Hs{3/2}}
          \norm{\partial_3\omega}{\Lq}\norm{\omega}{\Lq}^{q-1}.
\end{split}
\end{equation}
\end{lemma}
\begin{proof}
Set $a=2q/(2-q)$ and $q'=q/(q-1)$. Then
$1/2+1/a+1/q'=1$. On each horizontal plane, the two-dimensional Hardy--Littlewood--Sobolev inequality \cite{Stein1970} gives
\[
\norm{\partial_j\Delta_h^{-1}g}{L^a(\R^2)}
 \leq C_q\norm{g}{L^q(\R^2)}.
\]
Apply H\"older's inequality first on $\R^2_{x_h}$ with exponents $(2,a,q')$, and then on $\R_{x_3}$ with exponents $(\infty,q,q')$. The left-hand side of \eqref{eq:difficult} is bounded by
\[
C_q\norm{\partial_i v^3}{\mixed{\infty}{2}}
   \norm{\partial_3\omega}{\mixed{q}{q}}
   \norm{|\omega|^{q-1}}{\mixed{q'}{q'}}.
\]
The last two factors equal
$\norm{\partial_3\omega}{\Lq}$ and $\norm{\omega}{\Lq}^{q-1}$. Lemma~\ref{lem:trace} proves the claim.
\end{proof}

\paragraph{Vertical-vorticity energy.}
\begin{proposition}[Vertical-vorticity estimate]\label{prop:vorticity}
Let $v$ be a finite-energy solution of \eqref{eq:NS}, smooth in Sobolev spaces on a compact positive-time interval $J=[t_0,\tau]$. Assume that
\[
\omega\in C(J;W^{2,q}(\R^3))\cap C^1(J;L^q(\R^3)).
\]
Then, almost everywhere on $J$,
\begin{equation}\label{eq:X-energy}
\begin{aligned}
&\frac12\frac{\mathrm d}{\mathrm dt}\norm{\omega}{\Lq}^2
 +c_q\norm{\omega}{\Lq}^{2-q}\norm{\nabla\omega_{q/2}}{\Ltwo}^2\\
&\quad\leq C_q\norm{v^3}{\Hs{3/2}}\norm{\omega}{\Lq}
 \left(\norm{\omega}{\Lq}^{1-q/2}\norm{\nabla\omega_{q/2}}{\Ltwo}
 +\norm{\nabla^2v^3}{\Hh}\right).
\end{aligned}
\end{equation}
\end{proposition}
\begin{proof}
Taking the third component of the curl gives
\begin{equation}\label{eq:vorticity}
\partial_t\omega+v\cdot\nabla\omega-\Delta\omega
 =\partial_3v^3\,\omega+\partial_2v^3\,\partial_3v^1
                         -\partial_1v^3\,\partial_3v^2.
\end{equation}
Test by $\omega|\omega|^{q-2}$ and integrate over $\R^3$. The transport term vanishes by incompressibility. Integration by parts gives
\begin{equation}\label{eq:vorticity-identity}
\begin{split}
\frac1q\frac{\mathrm d}{\mathrm dt}\norm{\omega}{\Lq}^q
 +\frac{4(q-1)}{q^2}\norm{\nabla\omega_{q/2}}{\Ltwo}^2
 ={}&\int_{\R^3}\partial_3v^3\,|\omega|^q\dd x\\
 &+\int_{\R^3}
 (\partial_2v^3\,\partial_3v^1-\partial_1v^3\,\partial_3v^2)
 \omega|\omega|^{q-2}\dd x.
\end{split}
\end{equation}
Here and below the test is understood through a space--time regularization. More precisely, set
\[
\Phi_\eps(s)=\frac{(s^2+\eps^2)^{q/2}-\eps^q}{q},
\qquad
\Phi_\eps'(s)=s(s^2+\eps^2)^{q/2-1}.
\]
Use $\Phi_\eps'(\omega)$ multiplied by a smooth cutoff supported in a ball of radius $2R$, and integrate first over space and time. For $1<q<2$,
\[
0\leq\Phi_\eps(s)\leq\frac{|s|^q}{q},
\qquad |\Phi_\eps'(s)|\leq |s|^{q-1}.
\]
These bounds allow $R\to\infty$; in particular, the transport boundary term vanishes because the transport field is bounded and divergence free. Subtracting $\eps^q$ in $\Phi_\eps$ makes the regularized energy integrable on $\R^3$. For the diffusion term, \cite[Lemma 2.1]{HLLZ2019} gives
\[
\frac{4(q-1)}{q^2}\norm{\nabla\omega_{q/2}}{\Ltwo}^2
 =-\int_{\R^3}\Delta\omega\,\omega|\omega|^{q-2}\dd x
 \leq\norm{\Delta\omega}{\Lq}\norm{\omega}{\Lq}^{q-1}.
\]
Its hypotheses hold under the stated local regularity. Moreover,
$0\leq\Phi_\eps''(s)\leq C_q|s|^{q-2}$ for $s\neq0$, and $\nabla\omega=0$ almost everywhere on $\{\omega=0\}$. The displayed diffusion identity therefore provides an integrable majorant on $J\times\R^3$. Letting $\eps\downarrow0$ yields \eqref{eq:vorticity-identity} in integrated form and hence almost everywhere in time.

The stretching term satisfies
\begin{align}\label{eq:stretching}
\left|\int_{\R^3}\partial_3v^3\,|\omega|^q\dd x\right|
&\leq\norm{\partial_3v^3}{\Lthree}
       \norm{\omega_{q/2}}{\Lthree}^2\notag\\
&\leq C\norm{v^3}{\Hs{3/2}}\norm{\omega}{\Lq}^{q/2}
                          \norm{\nabla\omega_{q/2}}{\Ltwo},
\end{align}
by the Sobolev and Gagliardo--Nirenberg inequalities. To treat the remaining stretching terms without hiding any component, write the reconstruction formula componentwise:
\[
 v^1=-\partial_2\Delta_h^{-1}\omega
      -\partial_1\Delta_h^{-1}\partial_3v^3,
 \qquad
 v^2= \partial_1\Delta_h^{-1}\omega
      -\partial_2\Delta_h^{-1}\partial_3v^3.
\]
Differentiation in $x_3$ gives the exact decomposition
\begin{align}\label{eq:stretching-decomposition}
&\partial_2v^3\,\partial_3v^1-\partial_1v^3\,\partial_3v^2\notag\\
&\quad=-\partial_2v^3\,\partial_2\Delta_h^{-1}\partial_3\omega
       -\partial_1v^3\,\partial_1\Delta_h^{-1}\partial_3\omega\notag\\
&\qquad\phantom{=}
       -\partial_2v^3\,\partial_1\Delta_h^{-1}\partial_3^2v^3
       +\partial_1v^3\,\partial_2\Delta_h^{-1}\partial_3^2v^3.
\end{align}
The first two terms on the right of \eqref{eq:stretching-decomposition}, after multiplication by $\omega|\omega|^{q-2}$ and integration, are both bounded by Lemma~\ref{lem:difficult}. It remains to estimate the two terms containing $\partial_3^2v^3$. For arbitrary $i,j\in\{1,2\}$ set
\[
I_{\mathrm{div}}^{ij}
 =\int_{\R^3}\partial_i v^3\,
       \partial_j\Delta_h^{-1}\partial_3^2v^3\,
       \omega|\omega|^{q-2}\dd x.
\]
Choose $a_1,b$ by
\begin{equation}\label{eq:ab}
\frac1{a_1}=\frac1q-\frac\delta2=\frac12-\frac\delta6,
\qquad
\frac1b=\frac1q-\frac12=\frac\delta3.
\end{equation}
The horizontal and vertical H\"older relations are
\[
\frac1{a_1}+\frac\delta2+\frac1{q'}=1,
\qquad
\frac1b+\frac12+\frac1{q'}=1.
\]
Hence
\begin{equation}\label{eq:Idiv-holder}
|I_{\mathrm{div}}^{ij}|
 \leq\norm{\partial_i v^3}{\mixed{b}{a_1}}
 \norm{\partial_j\Delta_h^{-1}\partial_3^2v^3}{\mixed{2}{2/\delta}}
 \norm{\omega}{\Lq}^{q-1}.
\end{equation}
The two-dimensional Sobolev inequality and horizontal Riesz-transform boundedness give
\begin{equation}\label{eq:Idiv-second}
\norm{\partial_j\Delta_h^{-1}\partial_3^2v^3}{\mixed{2}{2/\delta}}
 \leq C_\delta\norm{\partial_3^2v^3}{\Hh}
 \leq C_\delta\norm{\nabla^2v^3}{\Hh}.
\end{equation}
For the first factor, horizontal and vertical Sobolev inequalities yield
\begin{align}\label{eq:Idiv-first}
\norm{\partial_i v^3}{\mixed{b}{a_1}}
&\leq C_\delta
  \norm{|\nabla_h|^{\delta/3}|\partial_3|^{1/2-\delta/3}
                \partial_i v^3}{\Ltwo}\notag\\
&\leq C_\delta\norm{v^3}{\Hs{3/2}}.
\end{align}
The first inequality is obtained by the horizontal Sobolev embedding on each plane and then the one-dimensional Sobolev embedding in $x_3$. The last step is a Fourier-weight comparison, since the total derivative order is $3/2$. Consequently, for every $i,j\in\{1,2\}$,
\begin{equation}\label{eq:Idiv}
|I_{\mathrm{div}}^{ij}|\leq C_q\norm{v^3}{\Hs{3/2}}
 \norm{\nabla^2v^3}{\Hh}\norm{\omega}{\Lq}^{q-1}.
\end{equation}
Combining the two curl terms and the two divergence terms in \eqref{eq:stretching-decomposition}, we obtain the complete bound
\begin{equation}\label{eq:all-cross-terms}
\begin{split}
&\left|\int_{\R^3}
 (\partial_2v^3\,\partial_3v^1-\partial_1v^3\,\partial_3v^2)
 \omega|\omega|^{q-2}\dd x\right|\\
&\quad\leq C_q\norm{v^3}{\Hs{3/2}}
 \left(\norm{\partial_3\omega}{\Lq}
       +\norm{\nabla^2v^3}{\Hh}\right)
 \norm{\omega}{\Lq}^{q-1}.
\end{split}
\end{equation}

Multiply \eqref{eq:vorticity-identity} by $\norm{\omega}{\Lq}^{2-q}$ and use \eqref{eq:G-W}, \eqref{eq:stretching}, and \eqref{eq:all-cross-terms}. Since
\[
\frac1q\norm{\omega}{\Lq}^{2-q}
       \frac{\mathrm d}{\mathrm dt}\norm{\omega}{\Lq}^{q}
 =\frac12\frac{\mathrm d}{\mathrm dt}\norm{\omega}{\Lq}^{2},
\]
we obtain
\eqref{eq:X-energy}.
The map $a\mapsto\norm{a}{\Lq}^2$ is continuously differentiable on $L^q(\R^3)$ for $q>1$, with derivative zero at $a=0$. Its Banach-space chain rule gives the norm identity also at times when $\omega=0$, with the dissipation convention in \eqref{eq:ED}. The resulting inequality holds almost everywhere and in its time-integrated form.
\end{proof}

\paragraph{Third-component estimate and coupling.}

For a divergence-free field $v\in H^m$, $m\geq3$, with $\omega\in W^{1,q}$, use its Riesz pressure $\pi_v$ from \eqref{eq:Pi}. 
When $v$ solves \eqref{eq:NS}, its third component satisfies
\[
\partial_t v^3+v\cdot\nabla v^3-\Delta v^3+\partial_3\pi_v=0.
\]
For each $k\in\{1,2,3\}$, apply $|\nabla_h|^{-\delta}\partial_k$ and test by $|\nabla_h|^{-\delta}\partial_kv^3$ in $L^2(\R^3)$. Summing in $k$ and integrating by parts gives
\begin{equation}\label{eq:Y-identity}
\frac12\frac{\mathrm d}{\mathrm dt}
 \norm{\nabla v^3}{\Hh}^2
 +\norm{\nabla^2v^3}{\Hh}^2
 =I_1(v)+I_2(v)+I_3(v),
\end{equation}
where the three spatial functionals are
\begin{equation}\label{eq:N}
\begin{aligned}
I_1(v)&=-\sum_{k=1}^3
 \pairh{\partial_kv\cdot\nabla v^3}{\partial_kv^3},\\
I_2(v)&=-\sum_{k=1}^3
 \pairh{v\cdot\nabla\partial_kv^3}{\partial_kv^3},\\
I_3(v)&=-\sum_{k=1}^3
 \pairh{\partial_k\partial_3\pi_v}{\partial_kv^3}.
\end{aligned}
\end{equation}
Thus $I_1$, $I_2$, and $I_3$ come from differentiated convection, transport, and pressure, respectively. The time-integrated identity is
\begin{equation}\label{eq:Y-integrated}
\begin{split}
\frac12\left(\norm{\nabla v^3(t)}{\Hh}^2
             -\norm{\nabla v^3(t_0)}{\Hh}^2\right)
 +\int_{t_0}^t\norm{\nabla^2v^3(s)}{\Hh}^2\dd s
 =\int_{t_0}^t\sum_{j=1}^3 I_j(v(s))\dd s.
\end{split}
\end{equation}

\begin{lemma}[Han--Lei--Li--Zhao]\label{lem:HLLZ}
Fix $q<2$ sufficiently close to $2$, and let $\delta$ be as in \eqref{eq:delta}. Suppose that $v\in H^m(\R^3)$ for an integer $m\geq3$, $\diver v=0$, and $\omega=\partial_1v^2-\partial_2v^1\in W^{1,q}(\R^3)$. With the Riesz pressure \eqref{eq:Pi}, the spatial functionals in \eqref{eq:N} satisfy
\begin{equation}\label{eq:HLLZ}
\begin{split}
|I_1(v)+I_2(v)+I_3(v)|
\leq{}& C_q\norm{v^3}{\Hs{3/2}}
 \left(\norm{\nabla v^3}{\Hh}+\norm{\omega}{\Lq}\right)\\
&\quad\times
 \left(\norm{\nabla^2v^3}{\Hh}+\norm{\nabla\omega}{\Lq}\right).
\end{split}
\end{equation}
\end{lemma}
We use the spatial estimate of \cite{HLLZ2019}. In Section 5 of the authors' preprint \arxiv{1708.04119v1}, $I_1$, $I_2$, and $I_3$ correspond, respectively, to (5.3), (5.4), and (5.5), summed over $k$; the four terms obtained by expanding \eqref{eq:HLLZ} are collected in (7.2) on page 25. The vorticity exponent called $r$ there is $q$ here. Spatial smoothing and the product estimates extend the bound to the stated finite-energy class. The estimate applies to the principal pairing of the remainder as well; Lemma~\ref{lem:common-perturbation} controls its background terms.

For clarity, put $V=\norm{v^3}{\Hs{3/2}}$, $X=\norm{\omega}{q}$,
$G=\norm{\nabla\omega}{q}$, $Y=\norm{\nabla v^3}{\Hh}$, and
$Z=\norm{\nabla^2v^3}{\Hh}$. The four terms of the cited estimate are
\[
 C_qV(YZ+YG+XZ+XG).
\]
They contain no logarithmic factor. The logarithmic vorticity estimate in the same reference is not used: that part is replaced by the vertical-vorticity estimate above. By \eqref{eq:G-W}, $G\leq C_qD^{1/2}$, whereas $X,Y\leq E^{1/2}$ and $Z\leq D^{1/2}$. It follows from \eqref{eq:Y-identity} that
\begin{equation}\label{eq:Y-energy}
\frac12\frac{\mathrm d}{\mathrm dt}
 \norm{\nabla v^3}{\Hh}^2
 +\norm{\nabla^2v^3}{\Hh}^2
 \leq C_q\norm{v^3}{\Hs{3/2}}E^{1/2}D^{1/2}.
\end{equation}
Proposition~\ref{prop:vorticity} gives the same right-hand side bound for the vorticity energy. Adding the two estimates yields
\begin{equation}\label{eq:coupled}
E'+cD\leq C\norm{v^3}{\Hs{3/2}}E^{1/2}D^{1/2}.
\end{equation}
The product with $D^{1/2}$ is the form used in the frequency splitting below.

The spatial bounds proved here also show, with the notation \eqref{eq:common-pairings}, that
\begin{equation}\label{eq:endpoint-principal}
 |\mathcal N_\omega(v)|+|\mathcal N_3(v)|
 \leq C_q\norm{v^3}{\Hs{3/2}}E^{1/2}D^{1/2}.
\end{equation}
Indeed, the vorticity proof bounds the absolute value of each principal term, and $\mathcal N_3(v)=-\sum_{j=1}^3 I_j(v)$. This is the form used for the heat-flow remainder.

\subsection{Frequency decomposition and Osgood control}\label{sec:frequency}
The high-frequency part is absorbed into the same coupled dissipation that occurs in \eqref{eq:endpoint-principal}. Only the finite middle range incurs a logarithmic summation loss.

\paragraph{Frequency decomposition.}
\begin{lemma}\label{lem:splitting}
Let $2<r\leq\infty$. For every integer $N\geq1$ and every scalar function $f$ with finite right-hand side,
\begin{equation}\label{eq:splitting}
\begin{split}
\norm{f}{\Hs{3/2}}
 \leq{}& C2^{-N(1/2+\delta)}\norm{\nabla f}{\Hh}
       +C(2N+1)^{1/2-1/r}\norm{f}{\Bs}\\
 &+C2^{-N(1/2-\delta)}\norm{\nabla^2f}{\Hh},
\end{split}
\end{equation}
where $1/\infty=0$.
\end{lemma}
\begin{proof}
Comparison of the Fourier weights gives
\begin{equation}\label{eq:layers}
\norm{f}{\Hs{1-\delta}}\leq\norm{\nabla f}{\Hh},
\qquad
\norm{f}{\Hs{2-\delta}}\leq\norm{\nabla^2f}{\Hh}.
\end{equation}
Indeed, $|\xi|^{2-2\delta}\leq|\xi_h|^{-2\delta}|\xi|^2$, and multiplication by $|\xi|^2$ gives the second comparison. Littlewood--Paley equivalence therefore implies
\begin{align}
\sum_{j<-N}2^{3j}\norm{\dot\Delta_jf}{\Ltwo}^2
 &\leq C2^{-2N(1/2+\delta)}\norm{\nabla f}{\Hh}^2,\label{eq:low}\\
\sum_{j>N}2^{3j}\norm{\dot\Delta_jf}{\Ltwo}^2
 &\leq C2^{-2N(1/2-\delta)}\norm{\nabla^2f}{\Hh}^2.\label{eq:high}
\end{align}
For the middle range, the finite-sequence H\"older inequality gives
\begin{equation}\label{eq:middle}
\begin{split}
\left(\sum_{|j|\leq N}2^{3j}\norm{\dot\Delta_jf}{\Ltwo}^2\right)^{1/2}
 &\leq(2N+1)^{1/2-1/r}
 \left(\sum_{|j|\leq N}
   \bigl(2^{3j/2}\norm{\dot\Delta_jf}{\Ltwo}\bigr)^r\right)^{1/r}\\
 &\leq(2N+1)^{1/2-1/r}\norm{f}{\Bs}.
\end{split}
\end{equation}
For $r=\infty$, the first right-hand side is interpreted using a supremum instead of the $\ell^r$ norm. Combining the low, middle, and high ranges proves \eqref{eq:splitting}.
\end{proof}

\paragraph{Osgood estimate.}
The following estimate includes the integrable errors from Lemma~\ref{lem:common-perturbation}. The logarithmic interpolation mechanism is related to the limiting inequalities in \cite{BG1980,BW1980,KT2000,KOT2002}; the frequency bound used here was proved above.
\begin{proposition}\label{prop:Osgood}
Let $2<r\leq\infty$ and $t_0<T<\infty$. Suppose that the measurable energy and dissipation defined in \eqref{eq:ED} satisfy
\[
E\in AC_{\mathrm{loc}}([t_0,T)),\qquad E\geq\mathrm e,
\qquad D\in L^1_{\mathrm{loc}}([t_0,T)),\qquad D\geq0,
\]
where $AC_{\mathrm{loc}}$ means absolute continuity on every compact interval $[t_0,\tau]\subset[t_0,T)$. Assume that, almost everywhere on $[t_0,T)$,
\begin{equation}\label{eq:general-energy}
E'+cD\leq C\norm{v^3}{\Hs{3/2}}E^{1/2}D^{1/2}+g(t)E,
\end{equation}
where $g\geq0$, $g\in L^1(t_0,T)$, and $E(t_0)<\infty$. If
$v^3\in L^2(t_0,T;\dot B^{3/2}_{2,r}(\R^3))$, then
\begin{equation}\label{eq:ED-bound}
\sup_{t_0\leq t<T}E(t)+\int_{t_0}^TD(t)\dd t<\infty.
\end{equation}
\end{proposition}
\begin{proof}
Set $\beta=1/2-\delta>0$ and choose, at each time,
\begin{equation}\label{eq:cutoff}
N(t)=\left\lceil N_0+K\log_2E(t)\right\rceil.
\end{equation}
Take $K\geq1/(2\beta)$ and then $N_0$ sufficiently large. Since
$\norm{\nabla v^3}{\Hh}\leq E^{1/2}$ and $E\geq e$, we have
\begin{equation}\label{eq:cutoff-properties}
C2^{-N\beta}E^{1/2}\leq\frac c4,
\qquad
2^{-N(1/2+\delta)}\norm{\nabla v^3}{\Hh}\leq1,
\qquad N\leq C\log E.
\end{equation}
The cutoff is chosen pointwise in the spatial estimate \eqref{eq:splitting}; it is not inserted into a time-dependent energy identity. Consequently no derivative of $N(t)$ occurs. In particular, this argument does not interchange the time norm and the dyadic summation.

Apply Lemma~\ref{lem:splitting} with $f=v^3$ in \eqref{eq:general-energy}. The high-frequency contribution satisfies
\begin{equation}\label{eq:absorption}
C2^{-N\beta}\norm{\nabla^2v^3}{\Hh}E^{1/2}D^{1/2}
 \leq C2^{-N\beta}E^{1/2}D
 \leq\frac c4D.
\end{equation}
After this absorption, the low- and middle-frequency contributions are bounded by
\[
C\left(1+\norm{v^3}{\Bs}(\log E)^{1/2-1/r}\right)E^{1/2}D^{1/2}.
\]
Young's inequality and $\log E\geq1$ now give
\begin{equation}\label{eq:log-energy}
E'+\frac c2D
 \leq C\left(1+\norm{v^3}{\Bs}^2+g\right)
 E(\log E)^{1-2/r}.
\end{equation}
The local absolute continuity of $E$ permits the chain rule on every $[t_0,\tau]$, $\tau<T$. For $2<r<\infty$, divide by $E(\log E)^{1-2/r}$ and integrate to obtain
\begin{equation}\label{eq:osgood-finite}
(\log E(t))^{2/r}
 \leq(\log E(t_0))^{2/r}
 +C_r\int_{t_0}^t\left(1+\norm{v^3(s)}{\Bs}^2+g(s)\right)\dd s.
\end{equation}
At $r=\infty$, the corresponding estimate is
\begin{equation}\label{eq:osgood-infty}
\log\log E(t)
 \leq\log\log E(t_0)
 +C\int_{t_0}^t\left(1+\norm{v^3(s)}{\Bs}^2+g(s)\right)\dd s.
\end{equation}
These are the two Osgood alternatives associated with
\[
\int_{\mathrm e}^\infty\frac{\dd z}{z(\log z)^\theta}=\infty
\qquad\text{for }\theta\leq1.
\]
They bound $E$ uniformly on $[t_0,T)$, with constants independent of the upper integration time $\tau<T$. A further integration of \eqref{eq:log-energy}, followed by $\tau\uparrow T$, bounds the integral of $D$ and proves \eqref{eq:ED-bound}.
\end{proof}

\subsection{Application to the heat-flow remainder}\label{sec:heat}
Let $u$ be the maximal $H^1$ strong solution from Theorem~\ref{thm:main}, and assume that $T^*<\infty$ and \eqref{eq:main-condition} holds. Since $H^1\hookrightarrow\dot H^{1/2}$, Lemmas~\ref{lem:common-remainder} and \ref{lem:common-perturbation} apply with the fixed choice of $q$. Form $U,v$ as in \eqref{eq:heat-decomposition}. In particular, $E(b)<\infty$ and the local energy tests are justified.

Combining \eqref{eq:endpoint-principal} with \eqref{eq:common-perturbed-energy} gives
\begin{equation}\label{eq:endpoint-perturbed}
 E'+cD\leq C\norm{v^3}{\Hs{3/2}}E^{1/2}D^{1/2}+gE,
 \qquad g\in L^1(b,T^*).
\end{equation}
The same fixed heat flow is bounded in $\dot B^{3/2}_{2,r}$ on $[b,T^*)$. Indeed, \eqref{eq:np-heat-coefficients}, together with the positive-time derivative bounds, controls its $\dot H^{3/2}$ norm, and $\dot H^{3/2}\hookrightarrow\dot B^{3/2}_{2,r}$ for $r\geq2$. Hence
\begin{equation}\label{eq:perturbed-Besov}
 \norm{v^3(t)}{\Bs}\leq\norm{u^3(t)}{\Bs}+\norm{U^3(t)}{\Bs}
\end{equation}
implies $v^3\in L^2(b,T^*;\dot B^{3/2}_{2,r})$. Proposition~\ref{prop:Osgood} therefore gives
\begin{equation}\label{eq:final-auxiliary}
 \sup_{b\leq t<T^*}E(t)+\int_b^{T^*}D(t)\dd t<\infty.
\end{equation}

\subsection{Completion of the endpoint proof}\label{subsec:endpoint-completion}
Set $\eta=1/2+\delta\in(0,1)$. Interpolation between the two Sobolev levels in \eqref{eq:layers} gives
\begin{equation}\label{eq:V-interpolation}
\norm{v^3}{\Hs{3/2}}
 \leq\norm{\nabla v^3}{\Hh}^{1-\eta}
      \norm{\nabla^2v^3}{\Hh}^{\eta}
 \leq E^{(1-\eta)/2}D^{\eta/2}.
\end{equation}
By H\"older's inequality in time and \eqref{eq:final-auxiliary},
\begin{equation}\label{eq:V-integral}
\begin{split}
\int_b^{T^*}\norm{v^3(t)}{\Hs{3/2}}^2\dd t
 &\leq\left(\sup_{b\leq t<T^*}E(t)\right)^{1-\eta}
 (T^*-b)^{1-\eta}
 \left(\int_b^{T^*}D(t)\dd t\right)^\eta<\infty.
\end{split}
\end{equation}
Adding back the heat flow gives
\begin{equation}\label{eq:u3-critical}
u^3\in L^2(b,T^*;\dot H^{3/2}(\R^3)).
\end{equation}

For the vertical vorticity, the three-dimensional Gagliardo--Nirenberg inequality and \eqref{eq:G-W} imply
\begin{equation}\label{eq:omega-interpolation}
\norm{\omega}{\Lthree}
 \leq C_q\norm{\omega}{\Lq}^{1-\eta}
             \norm{\nabla\omega}{\Lq}^{\eta}
 \leq C_qE^{(1-\eta)/2}D^{\eta/2}.
\end{equation}
Here $1/3=1/q-\eta/3$, so $\eta=3/q-1=1/2+\delta$. The same time estimate as in \eqref{eq:V-integral} gives
$\omega\in L^2(b,T^*;L^3(\R^3))$. Since $\omega_U$ is uniformly bounded in $L^3(\R^3)$, the original vertical vorticity satisfies
\begin{equation}\label{eq:omegau-critical}
\omega_u:=\partial_1u^2-\partial_2u^1=\omega+\omega_U
 \in L^2(b,T^*;L^3(\R^3)).
\end{equation}

Since $\norm{\nabla u^3}3\leq C\norm{u^3}{\dot H^{3/2}}$, equations \eqref{eq:u3-critical} and \eqref{eq:omegau-critical} make the coefficient in \eqref{eq:full-energy} integrable. Lemma~\ref{lem:full-gradient} and the basic kinetic energy bound yield
\begin{equation}\label{eq:full-bound}
 \sup_{b\leq t<T^*}\norm{u(t)}{H^1}^2
 +\int_b^{T^*}\norm{u(t)}{\dot H^2}^2\dd t<\infty.
\end{equation}
The lower-order terms are integrable on the finite time interval, so the heat bounds also give $v\in L^\infty(b,T^*;H^1)\cap L^2(b,T^*;H^2)$. Lemma~\ref{lem:terminal-trace} provides a strong terminal trace and continuation, with $h=u(a)\in H^1$. This proves Theorem~\ref{thm:main}.

\section{The nonendpoint Sobolev criterion}
\label{sec:nonendpoint}

This section proves Theorem~\ref{thm:nonendpoint}. Fix $2<p<\infty$ and put
\[
 s_p=\frac12+\frac2p.
\]
By rotation we take $\boldsymbol e=(0,0,1)$. The proof separates the fixed-time Sobolev estimate from the initial-data issue. The former is taken from the published work of Han, Lei, Li, and Zhao \cite{HLLZ2019}; the latter is handled by the common remainder and perturbation lemmas in Section~\ref{subsec:common-heat}.

\subsection{The fixed-time Sobolev estimate}
Fix $p>2$. We choose $q\in(3/2,2)$ in the near-$2$ range used in Sections~4--7 of \cite{HLLZ2019}, and set
\begin{equation}\label{eq:np-parameters}
 \delta=\frac3q-\frac32\in\left(0,\frac12\right).
\end{equation}
The admissible range is nonempty for every fixed finite $p$. When $2<p<4$, the construction in \cite{HLLZ2019} permits
$\max\{p/2,3/2\}<q<2$; when $p=4$, one takes $q<2$ sufficiently close to $2$. In both cases we may also impose
\begin{equation}\label{eq:np-q-condition-low-p}
 0<\frac2p+\frac3q-2<\frac2p,
\end{equation}
which is the condition used in the final interpolation step of that work. When $p>4$, the product estimates in Section~6 of \cite{HLLZ2019} involve a further auxiliary exponent $\varepsilon>0$. All of their exponent restrictions are strict and, at $(\delta,\varepsilon)=(0,0)$, have positive margins depending on $p$ (in particular on $1/4-1/p>0$). Hence one may first choose $q$ so close to $2$ that $\delta$ is sufficiently small and then choose $\varepsilon>0$ sufficiently small. This gives a common admissible choice for all estimates used below. We fix one such $q$ once and for all; constants may depend on $p$ and $q$, and no uniformity as $p\downarrow2$ or $p\uparrow\infty$ is asserted.

Use $X,W,Y,Z$ and the pairings from \eqref{eq:common-XYZW} and \eqref{eq:common-pairings}, with the present value of $q$. Write
\begin{equation}\label{eq:np-energy-notation}
 E_p=\mathrm e+X^2+Y^2,\qquad D_p=W^2+Z^2,\qquad
 A_p=\norm{w^3}{\dot H^{s_p}}.
\end{equation}
The subscript $p$ records the dependence of the chosen auxiliary exponent on $p$; it does not change the common energy formula.

\begin{lemma}[Fixed-time Sobolev estimate of Han--Lei--Li--Zhao]
\label{lem:np-principal}
For a smooth divergence-free field $w$ with finite quantities in \eqref{eq:np-energy-notation}, the pairings \eqref{eq:common-pairings} satisfy
\begin{equation}\label{eq:np-principal-bound}
 |\mathcal N_\omega(w)|+|\mathcal N_3(w)|
 \leq C_{p,q}A_pE_p^{1/p}D_p^{1-1/p}.
\end{equation}
The assertion is a spatial estimate at a fixed time and does not contain an assumption at the original initial time.
\end{lemma}

\begin{proof}
We spell out the normalization needed to pass from the cited estimates to \eqref{eq:np-principal-bound}. Put
\[
 F=\norm{\omega_{q/2}}2=X^{q/2},
 \qquad J=\norm{\nabla\omega_{q/2}}2,
 \qquad W=X^{1-q/2}J.
\]
Equations (7.8)--(7.10) of the preprint version \arxiv{1708.04119v1}, before their final applications of Young's inequality, estimate the vorticity and third-component pairings by $C A_p$ times combinations of
\[
 X^{2-q}F^{2/p}J^{2-2/p},
 \qquad
 X^{2-q}F^{1-2/q+2/p}J^{1-2/p}Z,
\]
\[
 Y^{2/p}Z^{2-2/p},
 \qquad
 \norm{\nabla\omega}{q}\,Y^{2/p}Z^{1-2/p},
\]
together with repetitions of the first two expressions. The first two expressions satisfy the exact identities
\begin{align}
 X^{2-q}F^{2/p}J^{2-2/p}
 &=X^{2/p}W^{2-2/p},\label{eq:np-normalization-one}\\
 X^{2-q}F^{1-2/q+2/p}J^{1-2/p}Z
 &=X^{2/p}W^{1-2/p}Z.\label{eq:np-normalization-two}
\end{align}
Moreover, \eqref{eq:G-W} bounds the last expression by a constant times
$Y^{2/p}Z^{1-2/p}W$. Consequently the cited pre-Young estimates give a constant multiple of $A_p$ times
\begin{equation}\label{eq:np-four-monomials}
 X^{2/p}W^{2-2/p}
 +X^{2/p}W^{1-2/p}Z
 +Y^{2/p}Z^{2-2/p}
 +Y^{2/p}Z^{1-2/p}W.
\end{equation}
The estimates in Sections~4--5 of \cite{HLLZ2019} give these terms for $2<p\leq4$, while Section~6 gives the corresponding anisotropic product and commutator bounds for $p>4$. Finally, weighted arithmetic--geometric mean gives, term by term,
\[
 X^{2/p}W^{2-2/p},\;X^{2/p}W^{1-2/p}Z,
 \;Y^{2/p}Z^{2-2/p},\;Y^{2/p}Z^{1-2/p}W
 \leq C E_p^{1/p}D_p^{1-1/p}.
\]
This proves \eqref{eq:np-principal-bound}. If $X=0$, the assertion follows by the standard regularization used in the $L^q$ vorticity test.
\end{proof}

\subsection{Heat-flow perturbation and Gronwall control}
Fix $0<a<b<T^*$ and form $U,v$ by \eqref{eq:heat-decomposition}. Lemma~\ref{lem:common-remainder} supplies the finite-energy estimate \eqref{eq:np-basic-remainder}, the required positive-time $L^q$ vorticity regularity, and local admissibility of $E_p,D_p$. These statements require no component condition. The full velocity need not be in $L^2$; all basic energy estimates are applied to $v$.

Lemma~\ref{lem:np-principal} bounds the principal terms in the common perturbed inequality \eqref{eq:common-perturbed-energy}. Thus
\begin{equation}\label{eq:np-coupled-energy}
 E_p'+cD_p\leq C A_p E_p^{1/p}D_p^{1-1/p}+gE_p,
 \qquad g\in L^1(b,T^*),
\end{equation}
where $A_p=\norm{v^3}{\dot H^{s_p}}$ and $g$ is given by \eqref{eq:common-g}. Young's inequality gives
\begin{equation}\label{eq:np-gronwall}
 E_p'+\frac c2D_p\leq C(A_p^p+g)E_p.
\end{equation}
All heat-flow transport, stretching, and pressure errors are already included in Lemma~\ref{lem:common-perturbation}; no new frequency decomposition is needed here.

The heat component satisfies
\begin{equation}\label{eq:np-heat-sobolev}
 \norm{U^3(t)}{\dot H^{s_p}}
 \leq C(t-a)^{-1/p}\norm{u(a)}{\dot H^{1/2}},\qquad t\geq b.
\end{equation}
Since $b>a$ and $T^*<\infty$, the right-hand side is in $L^p(b,T^*)$. The hypothesis \eqref{eq:nonendpoint-condition} therefore gives $A_p\in L^p(b,T^*)$. Applying Gronwall to \eqref{eq:np-gronwall} yields
\begin{equation}\label{eq:np-final-auxiliary}
 \sup_{b\leq t<T^*}E_p(t)+\int_b^{T^*}D_p(t)\dd t<\infty.
\end{equation}

\subsection{Completion of the nonendpoint proof}
Put $\eta=1/2+\delta\in(0,1)$. Since the horizontal negative norm dominates the isotropic one, interpolation and Sobolev embedding give
\begin{equation}\label{eq:np-L3-interpolation}
 \norm{\nabla v^3}3+\norm{\omega}3
 \leq C E_p^{(1-\eta)/2}D_p^{\eta/2}.
\end{equation}
For the vorticity term use $1/3=1/q-\eta/3$ and \eqref{eq:G-W}; for $\nabla v^3$, interpolate between $\dot H^{-\delta}$ and $\dot H^{1-\delta}$ to reach $\dot H^{1/2}$. H\"older's inequality in time and \eqref{eq:np-final-auxiliary}, followed by addition of the smooth heat flow, yield
\begin{equation}\label{eq:np-two-L3}
 \omega_u:=\partial_1u^2-\partial_2u^1\in L^2(b,T^*;L^3),
 \qquad
 \nabla u^3\in L^2(b,T^*;L^3).
\end{equation}

On every compact subinterval of $(b,T^*)$, the Fujita--Kato smoothing gives the hypotheses of Lemma~\ref{lem:homogeneous-gradient-identity}. Lemma~\ref{lem:full-gradient} therefore applies without an $L^2$ hypothesis on $u$. Since the coefficient in \eqref{eq:full-energy} belongs to $L^1(b,T^*)$ by \eqref{eq:np-two-L3}, Gronwall's inequality gives
\begin{equation}\label{eq:np-full-gradient-bound}
 \sup_{b\leq t<T^*}\norm{\nabla u(t)}2^2
 +\int_b^{T^*}\norm{\Delta u(t)}2^2\dd t<\infty.
\end{equation}
This does not assert that the full velocity belongs to $L^\infty L^2$. Combining \eqref{eq:np-full-gradient-bound}, the finite-energy remainder estimate \eqref{eq:np-basic-remainder}, and the heat bounds gives
\begin{equation}\label{eq:np-remainder-H1}
 \sup_{b\leq t<T^*}\norm{v(t)}{H^1}^2
 +\int_b^{T^*}\norm{v(t)}{H^2}^2\dd t<\infty.
\end{equation}
Lemma~\ref{lem:terminal-trace} now yields
\[
 \partial_tv\in L^2(b,T^*;L^2),\qquad
 v(t)\longrightarrow v_*\quad\text{strongly in }H^1
 \quad\text{as }t\uparrow T^*.
\]
Restarting the perturbed equation with datum $v_*$ continues $v$, and hence $u=U+v$, beyond $T^*$ in the Fujita--Kato class. This proves Theorem~\ref{thm:nonendpoint}.

\section*{Data availability}
No experimental or observational dataset is used in this mathematical study.

\section*{Declaration of generative AI and AI-assisted technologies}
During the preparation of this manuscript, OpenAI's ChatGPT was used to assist with language editing, LaTeX preparation, bibliographic checks, and preliminary review of the mathematical exposition and proposed proof steps. The authors are responsible for independently checking all arguments, citations, and statements in the final submitted version.

\end{document}